\documentclass[12pt]{amsart}

\makeatletter
  
  \@addtoreset{equation}{section}
\makeatother

\usepackage{amssymb,amsmath,amsthm}

\newcommand{\pd}{\partial}
\newcommand{\gm}{\gamma}

\newtheorem{thm}{Theorem}[section]
\newtheorem{lem}{Lemma}[section]
\newtheorem{prop}{Proposition}[section]

\newtheorem{definition}{Definition}[section]

\newcommand{\lm}{\lambda}

\newcommand{\R}{\mathbb{R}}

\newcommand{\N}{\mathbb{N}}

\newcommand{\av}[1]{\left| #1 \right|}

\newcommand{\vp}{\varphi}

\newcommand{\va}{\alpha}
\newcommand{\vb}{\beta}
\newcommand{\vd}{\delta}
\newcommand{\vz}{\zeta}

\newcommand{\mK}{\mathcal{K}}
\newcommand{\mC}{\mathcal{C}}

\newcommand{\mM}{\mathcal{M}}

\newcommand{\ve}{\varepsilon}

\newcommand{\con}[1]{\overline{#1}}

\title[The two obstacle problem for the parabolic biharmonic equation]{The two obstacle problem \\ for the parabolic biharmonic equation}

\author{M. Novaga} 
\address{Dipartimento di Matematica, Universit\`a di Pisa, Pisa, Italy}
\email{novaga@dm.unipi.it}

\author{S. Okabe}
\address{Mathematical Institute, Tohoku University, 980-8578 Sendai, Japan}
\email{okabe@math.tohoku.ac.jp}

\begin{document}

\begin{abstract}
We consider a two obstacle problem for the parabolic biharmonic equation 
in a bounded domain. 
We prove long time existence of solutions via an implicit time discretization 
scheme, and we investigate the regularity properties of solutions. 
\end{abstract} 
\subjclass{35K25, 35J35, 49A29}
\keywords{minimizing movements}
\thanks{The second author was partially supported by Grant-in-Aid for Young Scientists (B), No. 24740097, 
and by Strategic Young Researcher Overseas Visits Program for Accelerating Brain Circulation (JSPS)}

\maketitle

%%%%%%%%%%%%%%%%%%%%%%%%%%%%%%%%%%%%%%%%%%%%%%%%%%%%%%%%%%%
\section{Introduction}
The present paper is devoted to discussing a two obstacle problem 
for the parabolic biharmonic equation. 
The obstacle problem for second order elliptic and parabolic equations 
has attracted a great interest in the past years, and there is
an extensive mathematical literature 
(e.g., see \cite{C} and the references therein). On the contrary, 
much less is known on
the obstacle problem for higher order elliptic or parabolic equations. 

The biharmonic operator can be regarded as a prototype 
fourth order differential operator. 
Indeed, elliptic and parabolic PDEs for biharmonic operator 
are under intensive investigation in recent years
(see for example \cite{Ba, CM, CR, G, GG, GGS, GH, GP}).  
Although the obstacle problem for the biharmonic equation has been studied in the 1970s and 1980s (see \cite{BGT, BS, CF, CFT, F1,Schi}), 
some results on the obstacle problem for the corresponding parabolic equation 
have only been obtained very recently.
In particular, in \cite{NO} we considered the case of a single obstacle, i.e., the solution $u$ satisfies $u \ge f$ in $\Omega$ 
for a given obstacle function $f$ in a domain $\Omega$, and 
it is natural to ask whether the results can be extended to the case of two obstacles. 
Indeed, in this paper we prove the existence of solutions for the two obstacle problem,
and we investigate their regularity properties. 

Let $\Omega \subset \R^{N}$, with $N\le 3$, be a bounded domain with $\pd \Omega \in C^{4}$. 
Let $f : \Omega \to \R$ and $g : \Omega \to \R$ denote the obstacle functions satisfying 
\begin{gather}
f \in C^{4}(\con{\Omega}), \quad g \in C^4(\con{\Omega}), \quad f  \le g \quad \text{in} \quad \Omega, \label{fg-cond-1} \\
f < 0 < g \quad \text{on} \quad \pd \Omega.  \label{fg-cond-2}
\end{gather}
We consider a two obstacle problem of the type 
\begin{align} \label{P}
\begin{cases}
(\pd_{t} u + \Delta^{2} u)(u-f) \le 0 \qquad \,\,\,& \text{in} \quad \Omega \times \R_{+}, \\
 (\pd_{t} u + \Delta^{2} u)(u-g) \le 0 \qquad \,\,\,& \text{in} \quad \Omega \times \R_{+}, \\
 \pd_{t} u + \Delta^{2} u = 0 \qquad \,\,\,& \text{in} \quad \{\, (x,t) \in \Omega \times \R_{+} \mid f(x) < u(x,t) < g(x) \,\}, \\
 f \le u \le g \qquad \,\, \qquad & \text{in} \quad \Omega \times \R_{+}, \\
 u = \nabla u \cdot \nu^{\Omega} = 0 \,\,\,\,\quad & \text{on} \quad \pd \Omega \times \R_{+}, \\
 u(\cdot,0)=u_{0}(\cdot) \qquad \,\,\,\,\, & \text{in} \quad \Omega,   
\end{cases} \tag{P}
\end{align}
where $\nu^{\Omega}$ denotes the unit normal vector on $\pd \Omega$,
and the initial datum $u_{0} : \Omega \to \R$ satisfies 
\begin{align} \label{i-cond}
u_{0} \in H^{2}_{0}(\Omega), \quad f \le u_{0} \le g \quad \text{in} \quad \Omega. 
\end{align}
Here we define a weak solution of \eqref{P}. To this aim, we set 
\begin{align} \label{mc-K}
\mK &:= \{\, u \in L^2(0,T;H^2_0(\Omega)) \cap H^1(0,T; L^2(\Omega)) 
\mid u(x,0)= u_0(x) \,\,\,\text{a.e. in} \,\,\, \Omega, \\
& \qquad \qquad \qquad \qquad \qquad \qquad \qquad \qquad 
f(x) \le u(x,t) \le g(x) \,\,\, \text{a.e. in} \,\,\, \Omega \times (\, 0,T \,) \,\}. \notag
\end{align}
%%%%%%%%%%%%%%%%%%%%%%%%%
\begin{definition} \label{weak-sol}
We say that a function $u$ is a weak solution of \eqref{P} if 
\begin{enumerate}
\item[{\rm (i)}] $u \in \mK${\rm ;} 
\item[{\rm (ii)}] for any $w \in \mK$,  
\begin{align}
\int^T_0 \!\!\! \int_{\Omega} \left[ \pd_t u (w-u) + \Delta u \Delta (w-u) \right] \, dx dt \ge 0\,. 
\end{align}
\end{enumerate}
\end{definition}
%%%%%%%%%%%%%%%%%%%%%%%%%
Let us denote by $\Omega_0$ the coincidence set of $f$ and $g$, i.e., 
\begin{align} \label{Omega_0}
\Omega_0 = \{\, x \in \Omega \mid f(x)=g(x) \,\}. 
\end{align}
The main result of this paper is the following: 
%%%%%%%%%%%%%%%%%%%%%%%%%%%%%%%%%%%%
\begin{thm} \label{main-thm}
Let $N \le 3$. 
Let $f$ and $g$ satisfy \eqref{fg-cond-1}-\eqref{fg-cond-2}. 
Then, for any initial datum $u_0$ satisfying \eqref{i-cond}, the problem \eqref{P} possesses a unique weak solution 
\begin{align}
u \in L^{\infty}(\R_+ ; H^2_0(\Omega)) \cap H^1(\R_+ ; L^2(\Omega)).  
\end{align}
Moreover the quantity $\mu_t := \pd_t u(\cdot,t) + \Delta^2 u(\cdot, t)$ defines a signed measure in $\Omega$ for 
a.e. $t \in \R_+$, and for any $T>0$ there exists a constant $C>0$ such that 
\begin{align}
\int^T_0 \mu_t(\Omega)^2 \, dt < C + T \| \Delta^2 f \|^2_{L^{\infty}(\Omega_0)}. 
\end{align}
Furthermore the following regularity properties hold$\colon$  
\begin{enumerate}
\item[{\rm (i)}] $u \in L^2(\R_+ ; W^{2, \infty}(\Omega))$. In particular, if $N=1$, 
\begin{align}
u \in C^{0,\vb}(\R_+ ; C^{1, \gm}(\Omega)) \,\,\, \text{with} \,\,\, 0 < \gm < \dfrac{1}{2} \,\,\, \text{and} \,\,\, 
0 < \vb < \dfrac{1-2 \gm}{8}, 
\end{align}   
if $N \in  \{ 2, 3 \}$, 
\begin{align}
u \in C^{0,\vb}(\R_+ ; C^{0, \gm}(\Omega)) \,\,\, \text{with} \,\,\, 0 < \gm < \dfrac{4-N}{2} \,\,\, \text{and} \,\,\, 
0 < \vb < \dfrac{4-N-2 \gm}{8}; 
\end{align} 
\item[{\rm (ii)}] the signed measure $\mu_t$ satisfies 
\begin{gather}
\mu_t \lfloor_{\Omega_0} = \Delta^2 f,  \label{in-fg-coin} \\
{\rm supp}\, \mu_t \lfloor_{\Omega \setminus \Omega_0} \subset 
\{\, (x,t) \in (\Omega \setminus \Omega_0) \times \R_+ \mid u(x,t)= f(x) \,\,\, \text{or} \,\,\, u(x,t)=g(x) \,\}\,, \label{supp-mu_t}
\end{gather}
with
\begin{align}
\mu_t 
\begin{cases}
\ge 0 \quad & \text{in} \quad \{\, (x,t) \in (\Omega \setminus \Omega_0) \times \R_+ \mid u(x,t)= f(x) \,\} , \\
\le 0 \quad & \text{in} \quad \{\, (x,t) \in (\Omega \setminus \Omega_0) \times \R_+ \mid u(x,t)=g(x) \,\} . 
\end{cases}
\end{align}
In particular, $u$ satisfies \eqref{P} in the sense of distributions. 
\end{enumerate}
\end{thm}
%%%%%%%%%%%%%%%%%%%%%%%%%%%%%%%%%%%%

The restriction on the dimension $N\le 3$ in Theorem \ref{main-thm} has two motivations.
The first is related to the continuity of the approximate solutions. 
We construct the solution of \eqref{P} as a suitable limit of solutions of the 
obstacle problem 
for the corresponding elliptic equation, 
which is a biharmonic equation with a lower order perturbation. 
Here a difficulty arises from the presence of the set $\Omega_{0}$. 
To overcome this difficulty, first we construct the solution of the two obstacle problem replaced $f$ with $f - \ve$, for $\ve>0$. 
If the solution $u_{\ve}$ of the modified two obstacle problem is uniformly continuos with respect to $\ve$ in $\Omega$, 
then one can obtain a solution of the original obstacle problem as a limit of $u_{\ve}$ as $\ve \downarrow 0$. 
Thus the point is to obtain the uniform continuity of $u_{\ve}$,
and this is given by Sobolev's embedding if $N\le 3$. 
For the same reason, the two obstacle problem for the elliptic biharmonic equation was studied in \cite{CFT} 
under the same assumption $N \le 3$.

Even if $\Omega_0=\emptyset$, we still need the restriction on the dimension
in order to prove the $C^{1,1}$ regularity of the approximate solutions. 
Here the difficulty proving the continuity of the discrete velocities, 
which converge to $\pd_{t} u$. Again, such continuity can be obtained from 
Sobolev's embedding if $N\le 3$. 

We note that Theorem \ref{main-thm} can be extended to the problem \eqref{P} replaced Neumann boundary condition by Navier boundary 
condition, i.e., $u = \Delta u = 0$ on $\pd \Omega$. 
Indeed, replacing $H^{2}_{0}(\Omega)$ by $H^{2}(\Omega) \cap H^{1}_{0}(\Omega)$, 
we onbain the same conclusion as Theorem \ref{main-thm}

\smallskip

The paper is organized as follows: We shall construct the solution of \eqref{P} by way of an implicit time discretization so called 
minimizing movements, which was given by De Giorgi (e.g., see \cite{A}). 
We give a formulation via minimizing movement in Section \ref{setting}. 
In Section \ref{exist-appro-sol}, we construct an approximate solution of the problem \eqref{P} and investigate its regularity. 
In Section \ref{proof-of-main-theorem}, we prove Theorem \ref{main-thm}. 
Indeed, we first prove that the approximate solution converges to a function in a suitable sense. 
And then we observe that the limit is the required solution of \eqref{P}. 

%%%%%%%%%%%%%%%%%%%%%%%%%%%%%%%%%%%%%%%%%%%
%%%%%%%%%%%%%%%%%%%%%%%%%%%%%%%%%%%%%%%%%%%
%%%%%%%%%%%%%%%%%%%%%%%%%%%%%%%%%%%%%%%%%%%

\section{Notation} \label{setting}
We first note that the problem \eqref{P} is the $L^2$-gradient flow for the functional 
\begin{align}
E(u):= \dfrac{1}{2} \int_{\Omega} \av{\Delta u(x)}^2 \, dx 
\end{align}
with constraint $u \in \mK$.  
Let $T>0$ and $n \in \N$, and set $\tau_n = T/n$. We define a sequence $\{ u_{i,n} \}^{n}_{i=0}$ inductively. 
To begin with, we let $u_{0,n} := u_0$. 
Let us denote by $u_{i,n}$ the minimizer of the problem 
\begin{align} \label{mini_in}
\min \{\, G_{i,n}(u) \mid u \in K \,\}  \tag{$M_{i,n}$}
\end{align}
with 
\begin{align} \label{G_in}
G_{i,n}(u):= E(u) + P_{i,n}(u)
\end{align}
where 
\begin{align}
P_{i,n}(u):= \dfrac{1}{2 \tau_n} \int_{\Omega} [ u(x) - u_{i-1,n}(x) ] ^2\, dx. 
\end{align}
The set $K$ is given by 
\begin{align}
K = \{\, u \in H^2_0(\Omega) \mid f \le u \le g \,\,\, \text{in} \,\,\, \Omega \,\}. 
\end{align}
Let us set 
\begin{align}
V_{i,n}(x)= \dfrac{u_{i,n}(x) - u_{i-1,n}(x)}{\tau_n}. 
\end{align}

\begin{definition} \label{piece-linear}
Let us define $u_n(x,t) : \Omega \times [\, 0,T \,] \to \R$ as 
\begin{align}
u_n(x,t)= u_{i-1,n}(x) + (t - (i-1) \tau_n) V_{i,n}(x) 
\end{align}
in $\Omega \times [\, (i-1) \tau_n, i \tau_n \,]$ for each $i=1$, $2$, $\cdots$, $n$. 
\end{definition}

\begin{definition} \label{piece-constant}
Let us define $\tilde{u}_n(x,t) : \Omega \times (\, 0,T \,] \to \R$ and $V_n(x,t) : \Omega \times (\, 0,T \,] \to \R$ as 
\begin{align}
\tilde{u}_n(x,t)&= u_{i,n}(x), \\
V_n(x,t)&= V_{i,n}(x), 
\end{align}
in $\Omega \times (\, (i-1) \tau_n, i \tau_n \,]$ for each $i=1$, $2$, $\cdots$, $n$. 
\end{definition}

%%%%%%%%%%%%%%%%%%%%%%%%%%%%%%%%%%%%%%%%%%%%%%%%%%%%%%
%%%%%%%%%%%%%%%%%%%%%%%%%%%%%%%%%%%%%%%%%%%%%%%%%%%%%%
%%%%%%%%%%%%%%%%%%%%%%%%%%%%%%%%%%%%%%%%%%%%%%%%%%%%%%

\section{Existence of approximate solution} \label{exist-appro-sol}

To begin with, we show the existence of the solution of \eqref{mini_in}. 

%%%%%%%%%%%%%%%%%%%%%%%%%%%%%%%%%%%%%%%%%%%%%%%%%%%%%%
\begin{thm} \label{exist-mini}
Let $f$ and $g$ satisfy \eqref{fg-cond-1}-\eqref{fg-cond-2}.  Let $u_0$ satisfy \eqref{i-cond}. 
Then there exists a unique minimizer of \eqref{mini_in}. 
\end{thm}
%%%%%%%%%%%%%%%%%%%%%%%%%%%%%%%%%%%%%%%%%%%%%%%%%%%%%%
\begin{proof}
Let $\{ u_{j} \} \subset K$ be a minimizing sequence for the functional \eqref{G_in}. 
Since 
\begin{align*}
0 \le \inf_{K} G_{i,n}(u) \le G_{i,n}(u_{i-1,n}) = E(u_{i-1,n}),  
\end{align*}
we may assume $\{ u_j \}$ that $\sup_{j \in \N} G_{i,n}(u_j) < \infty$. 
Recalling that $\| \Delta v \|_{L^2(\Omega)}$ is equivalent to $\| v \|_{H^2_0(\Omega)}$ 
on $H^2_0(\Omega)$,
%(e.g., see \cite{L}), 
we deduce that $\{ u_j \}$ is uniformly bounded in $H^2_0(\Omega)$, 
and then there exists $u \in H^{2}_{0}(\Omega)$ such that 
\begin{align} \label{weak-conv-311}
u_j \rightharpoonup u \quad \text{in} \quad H^2(\Omega), 
\end{align}
in particular,  
\begin{align} \label{weak-conv-312}
\Delta u_j \rightharpoonup \Delta u \quad \text{in} \quad L^2(\Omega), 
\end{align}
up to a subsequence. 
Since \eqref{weak-conv-311} implies that $u_j$ uniformly converges to $u$ in $\Omega$ up to a subsequence, 
we have $f \le u \le g$ in $\Omega$. 
It follows from Fatou's Lemma that 
\begin{align*}
P_{i,n}(u) \le \liminf_{j \to \infty} P_{i,n}(u_j). 
\end{align*}
Moreover we infer from \eqref{weak-conv-312} that 
\begin{align*}
E(u) \le \liminf_{j \to \infty} E(u_j). 
\end{align*}
The uniqueness of the minimizer of \eqref{mini_in} follows from the convexity of $G_{i,n}$. 
\end{proof}

Set 
\begin{align}
f_\ve(x) = f(x) - \ve. 
\end{align}
We denote by $(M^{\ve}_{i,n})$ the problem \eqref{mini_in} replaced $f$ by $f_\ve$. 
The proof of Theorem \ref{exist-mini} implies that the problem $(M^{\ve}_{i,n})$ has a unique minimizer $u^{\ve}_{i,n}$. 
From now on, let us set 
\begin{align}
V^{\ve}_{i,n} = \dfrac{u^{\ve}_{i,n} - u^{\ve}_{i-1,n}}{\tau_n}. 
\end{align}
Moreover let $V^{\ve}_{n}$ denote the piecewise constant interpolation of $V^{\ve}_{i,n}$. 

%%%%%%%%%%%%%%%%%%%%%%%%%%%%%%%%%%%%%%%%%%%%%%%%%%%%%%
\begin{lem} \label{uniform-conv-1}
$u^{\ve}_{i,n}$ uniformly converges to $u_{i,n}$ in $\Omega$ as $\ve \to 0$. 
\end{lem}
%%%%%%%%%%%%%%%%%%%%%%%%%%%%%%%%%%%%%%%%%%%%%%%%%%%%%%
\begin{proof}
By the fact that $\| u^{\ve}_{i,n} \|_{H^2(\Omega)} \le C$, for any sequence $\{ \ve_m \}$ with $\ve_m \to 0$ as $m \to \infty$, 
there exist $\{ \ve_{m'} \} \subset \{ \ve_m \}$ 
and $\bar{u}_{i,n} \in H^2_0(\Omega)$ such that 
\begin{align} \label{weak-u-ve-1}
u^{\ve_{m'}}_{i,n} \rightharpoonup \bar{u}_{i,n} \quad \text{weakly in} \quad H^2(\Omega) \quad \text{as} \quad m' \to \infty,  
\end{align}
in particular, 
\begin{align} \label{weak-u-ve-2}
\Delta u^{\ve_{m'}}_{i,n} \rightharpoonup \Delta \bar{u}_{i,n} 
\quad \text{weakly in} \quad L^2(\Omega) \quad \text{as} \quad m' \to \infty. 
\end{align}
Since $N \le 3$, Sobolev's embedding theorem implies that $u^{\ve_{m'}}_{i,n}$ uniformly converges to $\bar{u}_{i,n}$ as 
$\ve \downarrow 0$. 
Recalling that the solution $u^{\ve_{m'}}_{i,n}$ of $(M^{\ve_{m'}}_{i,n})$ satisfies 
\begin{align*} 
\int_\Omega \left[ \Delta u^{\ve_{m'}}_{i,n} \Delta (w-u^{\ve_{m'}}_{i,n}) + V^{\ve_{m'}}_{i,n} (w-u^{\ve_{m'}}_{i,n}) \right] \, dx \ge 0 
\quad \text{for any} \quad w \in K_{\ve_{m'}}, 
\end{align*}
we deduce from \eqref{weak-u-ve-1}-\eqref{weak-u-ve-2} that    
\begin{align*}
&\int_\Omega \left[ \Delta \bar{u}_{i,n} \Delta (w-\bar{u}_{i,n}) + \bar{V}_{i,n} (w-\bar{u}_{i,n}) \right] \, dx \\
& \quad \ge \liminf_{m' \to \infty} 
    \int_\Omega \left[ \Delta u^{\ve_{m'}}_{i,n} \Delta (w-u^{\ve_{m'}}_{i,n}) + V^{\ve_{m'}}_{i,n} (w-u^{\ve_{m'}}_{i,n}) \right] \, dx
\ge 0 
\quad \text{for any} \quad w \in K,  
\end{align*}
where we used the fact $K \subset K_{\ve_{m'}}$. 
Moreover it follows from the uniqueness of the solution of \eqref{mini_in} that $\bar{u}_{i,n}= u_{i,n}$. 
\end{proof}

Along the same lines as in the proof of Theorem 2.2  in \cite{NO}, we obtain the following uniform estimates: 
%%%%%%%%%%%%%%%%%%%%%%%%%%%%%%%%%%%%%%%%%%%%%%%%%%%%%%
\begin{prop} \label{est-mini}
Let $u^{\ve}_{i,n}$ be the solution of $(M^{\ve}_{i,n})$. Then, for any $n \in \N$, 
\begin{gather}
\int^T_0 \!\!\! \int_\Omega V^{\ve}_{n}(x,t)^2 \, dxdt \le 2 E(u_0), \label{bdd-V_n} \\
\sup_i \| \Delta u^{\ve}_{i,n} \|_{L^2(\Omega)}^2 \le 2 E(u_0). \label{bdd-Delta-u_in}
\end{gather}
\end{prop}
%%%%%%%%%%%%%%%%%%%%%%%%%%%%%%%%%%%%%%%%%%%%%%%%%%%%%%

Since $N \le 3$, combining Proposition \ref{est-mini} with Sobolev's embedding theorem, we have 
\begin{align} \label{uni-conti-in-ve}
&\text{$u^{\ve}_{i,n}$ is uniformly continuous in $\Omega$, with modulus of continuity} \\
&\text{independent of $\ve$, $i$, and $n$.} \notag
\end{align}
Set 
\begin{align*}
\mC^{\ve, +}_{i,n} &= \{\, x \in \Omega \mid u^{\ve}_{i,n}(x)= f_\ve(x) \,\}, \\
\mC^{\ve, -}_{i,n} &= \{\, x \in \Omega \mid u^{\ve}_{i,n}(x)= g(x) \,\}. 
\end{align*}
By the fact that $f_\ve < g$ in $\Omega$, we observe from \eqref{uni-conti-in-ve} that 
the sets $\mC^{\ve,+}_{i,n}$ and $\mC^{\ve,-}_{i,n}$ are disjoint. Here we set   
\begin{align*}
\mu^{\ve}_{i,n}= \Delta^2 u^{\ve}_{i,n} + V^{\ve}_{i,n}. 
\end{align*}
In the following, we show that $\mu^{\ve}_{i,n}$ is a signed measure in $\Omega$. 
To this aim, let us define 
\begin{align*}
\gm_\rho(\lm)&:=
\begin{cases}
\dfrac{\lm^2}{\rho} \quad & \text{if} \quad \lm <0, \\
0 \quad & \text{if} \quad \lm >0,
\end{cases}\\
\vb_\rho(\lm) &:= \gm'_\rho(\lm),  
\end{align*}
for each $\rho>0$. Regarding the following minimization problem 
\begin{align} \label{mini-in-ve-rho}
\min_{v \in H^2_0(\Omega)} G^{\ve,\rho}_{i,n}(v)  \tag{$M^{\ve,\rho}_{i,n}$}
\end{align}
with 
\begin{align*}
G^{\ve, \rho}_{i,n}(v) := \int_\Omega \Bigm[ \dfrac{1}{2}(\Delta v)^2 + \dfrac{1}{2 \tau_n} (v-u^{\ve}_{i-1,n})^2 
                                                                  + \gm_\rho(v-f_\ve) + \gm_\rho(g-v) \Bigm] \, dx, 
\end{align*}
we show the following: 
%%%%%%%%%%%%%%%%%%%%%%%%%%%%%%%%%%%%%%%%%%%%%%%%%%%%%%
\begin{prop} \label{appro-u-in-ve}
The problem \eqref{mini-in-ve-rho} has a unique solution $w^{\ve,\rho}_{i,n}$ with 
\begin{align}
w^{\ve,\rho}_{i,n} \rightharpoonup u^{\ve}_{i,n} \quad \text{weakly in} \quad H^2(\Omega) 
\quad \text{as} \quad \rho \downarrow 0. 
\end{align}
\end{prop}
%%%%%%%%%%%%%%%%%%%%%%%%%%%%%%%%%%%%%%%%%%%%%%%%%%%%%%
\begin{proof}
By a standard argument, we deduce that the problem \eqref{mini-in-ve-rho} has a unique solution $w^{\ve,\rho}_{i,n}$ satisfying 
\begin{align*}
\Delta^{2} w^{\ve,\rho}_{i,n} + \dfrac{1}{\tau_{n}} (w^{\ve,\rho}_{i,n} - u^{\ve}_{i-1,n}) 
  + \vb_{\rho}(w^{\ve,\rho}_{i,n} - f_{\ve}) - \vb_{\rho}(g - w^{\ve,\rho}_{i,n}) = 0 \quad \text{in} \quad \Omega
\end{align*}
in the classical sense. Since it follows from the minimality of $w^{\ve,\rho}_{i,n}$ that 
\begin{align} \label{Gw-Eu}
G^{\ve,\rho}_{i,n}(w^{\ve,\rho}_{i,n}) 
\le G^{\ve,\rho}_{i,n}(u^{\ve}_{i-1,n}) 
= E(u^{\ve}_{i-1,n}), 
\end{align}
we observe from Proposition \ref{est-mini} that 
\begin{gather}
\| \Delta w^{\ve,\rho}_{i,n} \|_{L^{2}(\Omega)}^{2} \le 2 E(u_{0}), \label{H2-bdd-w} \\
\dfrac{1}{2 \tau_{n}} \| w^{\ve,\rho}_{i,n} - u^{\ve}_{i-1,n} \|_{L^{2}(\Omega)}^{2} \le E(u_{0}), 
\end{gather}
and 
\begin{align} \label{w-n-part}
\max \{\, \| (w^{\ve,\rho}_{i,n} - f_\ve)^{-} \|_{L^{2}(\Omega)}^{2}, \| (g - w^{\ve,\rho}_{i,n})^{-} \|_{L^{2}(\Omega)}^{2} \,\} 
\le \rho E(u_{0}). 
\end{align}
The inequality \eqref{H2-bdd-w} yields that there exist a sequence $\{ \rho_{m}\}$ with $\rho_{m} \to 0$ as $m \to \infty$ and 
a function $\tilde{u} \in H^{2}_{0}(\Omega)$ such that 
\begin{align} \label{w-weak}
w^{\ve,\rho_{m}}_{i,n} \rightharpoonup \tilde{u} \quad \text{weakly in} \quad H^{2}(\Omega), 
\end{align}
in particular, 
\begin{align} \label{w-ae-conv}
w^{\ve,\rho_{m}}_{i,n} \to \tilde{u} \quad \text{a.e. in}\quad \Omega, 
\end{align}
as $\rho_{m} \to 0$. Recalling \eqref{w-n-part} and \eqref{w-ae-conv}, we deduce from Chebyshev's inequality that 
$f_{\ve} \le \tilde{u} \le g$ in $\Omega$. This implies $\tilde{u} \in K_{\ve}$. 

We claim that $\tilde{u}$ is a minimizer of $(M^{\ve}_{i,n})$. 
Indeed, for any $v \in K_{\ve}$, it holds that 
\begin{align} \label{pre-mini-1}
G^{\ve}_{i,n}(v) = G^{\ve,\rho_{m}}_{i,n}(v) 
\ge G^{\ve,\rho_{m}}_{i,n}(w^{\ve,\rho_{m}}_{i,n}) 
\ge G^{\ve}_{i,n}(w^{\ve,\rho_{m}}_{i,n}). 
\end{align}
Recalling \eqref{w-weak}-\eqref{w-ae-conv} and letting $\rho_{m} \to 0$ in \eqref{pre-mini-1}, we infer that 
\begin{align*} 
G^{\ve}_{i,n}(v) \ge \liminf_{\rho_{m} \downarrow 0} G^{\ve}_{i,n}(w^{\ve,\rho_{m}}_{i,n}) = G^{\ve}_{i,n}(\tilde{u}). 
\end{align*}
This implies that $\tilde{u}$ is a minimizer of $(M^{\ve}_{i,n})$. 
Then it follows from the uniqueness of the solutions to $(M^{\ve}_{i,n})$ that $\tilde{u}= u^{\ve}_{i,n}$. 
We thus completed the proof. 
\end{proof}

%%%%%%%%%%%%%%%%%%%%%%%%%%%%%%%%%%%%%%%%%%%%%%%%%%%%%%
\begin{thm} \label{Radon-1}
Let $\ve>0$ and $i \in \{\, 1, 2, \cdots, n \,\}$. 
Then the quantity $\mu^{\ve}_{i,n}$ is a signed measure in $\Omega$ with 
\begin{align} \label{prop-mu-in-ve}
{\rm supp}\, \mu^{\ve}_{i,n} \subset \mC^{\ve,+}_{i,n} \cup \mC^{\ve,-}_{i,n}, \qquad 
\mu^{\ve}_{i,n} 
\begin{cases}
\ge 0 \quad & \text{in} \quad \mC^{\ve,+}_{i,n}, \\
\le 0 \quad & \text{in} \quad \mC^{\ve,-}_{i,n}. 
\end{cases}
\end{align} 
Moreover there exists a positive constant $C>0$ independent of $\ve$ and $n$ such that 
\begin{align} \label{bdd-meas-1}
\tau_n \sum^{n}_{i=1} \mu^{\ve}_{i,n}(\Omega)^2 < C. 
\end{align}
\end{thm}
%%%%%%%%%%%%%%%%%%%%%%%%%%%%%%%%%%%%%%%%%%%%%%%%%%%%%%
\begin{proof}
To begin with, we shall verify that the quantity 
\begin{align*}
\mu^{\ve,\rho}_{i,n}:= \Delta^2 w^{\ve,\rho}_{i,n} + (w^{\ve,\rho}_{i,n} - u^{\ve}_{i-1,n})/\tau_n
\end{align*}
defines a signed measure in $\Omega$. Let us set 
\begin{align*}
I_{\rho}^{+} = \{\, x \in \Omega \mid w^{\ve,\rho}_{i,n}(x) \le f_\ve(x) \,\}, \quad 
I_{\rho}^{-} = \{\, x \in \Omega \mid w^{\ve,\rho}_{i,n}(x) \ge g(x) \,\}. 
\end{align*}
It follows from $\vb_\rho \le 0$ that 
\begin{align*}
\Delta^2 w^{\ve,\rho}_{i,n} + \dfrac{w^{\ve,\rho}_{i,n} - u^{\ve}_{i-1,n}}{\tau_n} 
 = -\vb_\rho(w^{\ve,\rho}_{i,n} - f_{\ve}) + \vb_\rho(g - w^{\ve,\rho}_{i,n}) 
\begin{cases}
\ge 0 \quad & \text{in} \quad I_{\rho}^{+}, \\
= 0 \quad & \text{in} \quad \Omega \setminus (I_{\rho}^{+} \cup I_{\rho}^{-}), \\
\le 0 \quad & \text{in} \quad I_{\rho}^{-},  \\
\end{cases} 
\end{align*}
i.e., $\mu^{\ve,\rho}_{i,n}$ defines a signed measure in $\Omega$. 

We claim that the measure $\mu^{\ve,\rho}_{i,n}$ converges to $\mu^{\ve}_{i,n}$ as $\rho \downarrow 0$ 
up to a subsequence. Indeed, we shall show that, for each $\ve$, $i$, and $n$, the quantity 
$\mu^{\ve,\rho}_{i,n}(U)$ is uniformly bounded with respect to $\rho$ for any $U \subset \subset \Omega$.   
From now on, we write $\mu^{\ve,\rho}_{i,n}= \nu^{\rho}_+ - \nu^{\rho}_-$, where $\nu^{\rho}_{\pm}$ are positive 
measures with their support in $I^{\pm}_{\rho}$, respectively. 
For any $\vp \in C^\infty_c(\Omega)$ with $\vp \equiv 1$ in $U$ and $0 \le \vp \le 1$ elsewhere, we observe that 
\begin{align*}
\nu^{\rho}_{\pm}(U)  
&\le \int_{U} \vp d \nu^{\rho}_{\pm} 
= \pm \int_{U} 
       \Bigm[ \Delta w^{\ve,\rho}_{i,n} \Delta \vp + \dfrac{1}{\tau_n}(w^{\ve,\rho}_{i,n} - u^{\ve}_{i-1,n}) \vp \Bigm]\, dx \\
& \le E(w^{\ve,\rho}_{i,n})^{{\frac{1}{2}}} E(\vp)^{\frac{1}{2}} 
 + \sqrt{\dfrac{2}{\tau_n}} 
       \left( \dfrac{1}{2 \tau_n} \int_\Omega (w^{\ve,\rho}_{i,n} - u^{\ve}_{i-1,n})^2 \, dx \right)^{\frac{1}{2}}
       \left( \int_\Omega \vp^2 \, dx \right)^{\frac{1}{2}}. 
\end{align*}
Since it follows from \eqref{Gw-Eu} that 
\begin{align} \label{P-E-diff}
\dfrac{1}{2 \tau_n} \int_\Omega (w^{\ve,\rho}_{i,n} - u^{\ve}_{i-1,n})^2 \, dx 
  \le E(u^{\ve}_{i-1,n}) - E(w^{\ve, \rho}_{i,n}), 
\end{align}
we observe from \eqref{H2-bdd-w} and \eqref{P-E-diff} that   
\begin{align} \label{pre-measu-bdd-1-1}
\nu^{\rho}_{\pm}(U) 
 \le C(U) \Bigm[ E(u_0)^{\frac{1}{2}} + \Bigm( \dfrac{E(u^{\ve}_{i-1,n}) - E(w^{\ve, \rho}_{i,n})}{\tau_n} \Bigm)^{\frac{1}{2}} \Bigm]. 
\end{align}
Thus there exist a sequence $\{ \rho_{m'} \} \subset \{ \rho_m \}$ and measures $\bar{\mu}_{\pm}$ such that 
\begin{align} \label{nu-to-mu-1}
\nu^{\rho_{m'}}_{\pm} \rightharpoonup \bar{\mu}_{\pm} \quad \text{as} \quad m' \to \infty,  
\end{align}
i.e., for any $\vz \in C_c(\Omega)$ 
\begin{align*}
\int_\Omega \vz d \nu^{\rho_{m'}}_{\pm} \to \int_\Omega \vz d \bar{\mu}_{\pm}  \quad \text{as} \quad m' \to \infty,  
\end{align*}
where $\{ \rho_m \}$ is the sequence obtained in the proof of Proposition \ref{appro-u-in-ve}. 
Since Proposition \ref{appro-u-in-ve} asserts that 
\begin{align*}
\int_\Omega \vz d \nu^{\rho_{m'}}_{\pm} 
& =  \pm \int_\Omega  
     \Bigm[ \Delta w^{\ve,\rho_{m'}}_{i,n} \Delta \vz + \dfrac{1}{\tau_n}(w^{\ve,\rho_{m'}}_{i,n} - u^{\ve}_{i-1,n}) \vz \Bigm]\, dx \\
&\to  \pm \int_\Omega  
     \left[ \Delta u^{\ve}_{i,n} \Delta \vz + V^{\ve}_{i,n} \vz \right]\, dx 
     \quad \text{for any} \quad \vz \in C^2_c(\Omega) \quad \text{as} \quad m' \to \infty, 
\end{align*}
the relation \eqref{nu-to-mu-1} implies $\bar{\mu}_{\pm}= \pm(\Delta^2 u^{\ve}_{i,n} + V^{\ve}_{i,n})$, respectively. 
We claim that 
\begin{align} \label{supp-mu-in-ve}
{\rm supp}\, \bar{\mu}_{+} \subset \mC^{\ve,+}_{i,n}, \qquad {\rm supp}\, \bar{\mu}_{-} \subset \mC^{\ve,-}_{i,n}. 
\end{align}
It is sufficient to show the former relation. 
Let $x_0 \in \Omega \setminus \mC^{\ve,+}_{i,n}$ be chosen arbitrarily. Then there exist a neighborhood $W$ of $x_0$ and a 
constant $\vd>0$ such that 
\begin{align*}
u^{\ve}_{i,n}(x) - f_\ve(x) > \vd \quad \text{in} \quad W \subset \Omega.
\end{align*}
Since $w^{\ve, \rho_{m'}}_{i,n}$ uniformly converges to $u^{\ve}_{i,n}$ as $m' \to \infty$, there exists a constant 
$M>0$ such that for any $m' > M$ 
\begin{align*}
\av{w^{\ve,\rho_{m'}}_{i,n} - u^{\ve}_{i,n}} \le \dfrac{\vd}{2} \quad \text{in} \quad W. 
\end{align*}
Thus we deduce that, for any $m'>M$,  
\begin{align*}
w^{\ve,\rho_{m'}}_{i,n}(x) - f_\ve(x) 
 \ge (u^{\ve}_{i,n} - f_\ve(x)) - \av{w^{\ve,\rho_{m'}}_{i,n} - u^{\ve}_{i,n}} 
 > \dfrac{\vd}{2},  
\end{align*}
i.e., $W \subset \Omega \setminus I^+_{\rho_{m'}}$ for any $m' > M$. Hence we see that for any $\vz \in C^2_c(W)$ 
\begin{align*}
\int_\Omega \vz \, d \bar{\mu}_+ 
 = \lim_{m' \to \infty} \int_\Omega \vz d \nu^{\rho_{m'}}_{+}  =0.      
\end{align*}
This is equivalent to the former relation in \eqref{supp-mu-in-ve}. 
Recalling that $\mC^{\ve,+}_{i,n}$ and $\mC^{\ve,-}_{i,n}$ are disjoint set, we observe that $\mu^{\ve}_{i,n}$ is a signed 
measure satisfying \eqref{prop-mu-in-ve}. 

We turn to the proof of \eqref{bdd-meas-1}. 
For any $U \subset \subset \Omega$, it follows from \eqref{pre-measu-bdd-1-1} that 
\begin{align*}
\mu^{\ve}_{i,n} \le \mu^{\ve}_{i,n} \lfloor_{\mC^{\ve,+}_{i,n}} 
& \le C(U) E(u_0)^{\frac{1}{2}} 
      + C(U) \liminf_{\rho_{m'} \downarrow 0} \Bigm( \dfrac{E(u^{\ve}_{i-1,n}) - E(w^{\ve, \rho_{m'}}_{i,n})}{\tau_n} \Bigm)^{\frac{1}{2}} \\
& = C(U) E(u_0)^{\frac{1}{2}} 
      + C(U) \Bigm( \dfrac{E(u^{\ve}_{i-1,n}) - E(u^{\ve}_{i,n})}{\tau_n} \Bigm)^{\frac{1}{2}}
\end{align*}
and 
\begin{align*}
\mu^{\ve}_{i,n} \ge \mu^{\ve}_{i,n} \lfloor_{\mC^{\ve,-}_{i,n}} 
& \ge -C(U) E(u_0)^{\frac{1}{2}} 
      - C(U) \liminf_{\rho_{m'} \downarrow 0} \Bigm( \dfrac{E(u^{\ve}_{i-1,n}) - E(w^{\ve, \rho_{m'}}_{i,n})}{\tau_n} \Bigm)^{\frac{1}{2}} \\
& = -C(U) E(u_0)^{\frac{1}{2}} 
      - C(U) \Bigm( \dfrac{E(u^{\ve}_{i-1,n}) - E(u^{\ve}_{i,n})}{\tau_n} \Bigm)^{\frac{1}{2}}. 
\end{align*}
Multiplying $\tau_n$ and summing over $i=0$, $1$, $\cdots$, $n$, we find 
\begin{align} \label{pre-mu-in-ve-bdd}
\tau_n \sum^{n}_{i=0} \mu^{\ve}_{i,n}(U)^2 \le C'(U) E(u_0) T + C'(U) (E(u_0) - E(u_{n,n})) \le C'(U)(T+1)E(u_0).  
\end{align}

It follows from the condition \eqref{fg-cond-2} that there exists a constant $\vd_*>0$ such that 
\begin{align*} %\label{fg-away-boundary}
d(\pd \Omega, \mC^{\ve,\pm}_{i,n}) \ge \vd_*. 
\end{align*}
Thus it follows from \eqref{prop-mu-in-ve} that ${\rm supp}\, \mu^{\ve}_{i,n} \subset \Omega_{\vd_{*}/2}$, where 
$\Omega_{\rho}:= \{ x \in \Omega \mid {\rm dist}(x, \pd \Omega) > \rho \}$. 
Letting $U= \Omega_{\vd_{*}/2}$, we obtain the conclusion. 
\end{proof}

We shall now prove the $C^{1,1}$ regularity of $u^{\ve}_{i,n}$ in $\Omega$. 
In the following, for each $h \in L^{2}(\Omega)$, we denote by $\Delta^{-1} h$ the solution of 
\begin{align*}
\begin{cases}
- \Delta w = h & \text{in} \quad \Omega, \\
w = 0 & \text{on} \quad \pd \Omega. 
\end{cases}
\end{align*}
We start with the following lemma: 
%%%%%%%%%%%%%%%%%%%%%%%%%%%%%%%%%%%%%%%%%%%%%%%%%%%%%%
\begin{lem} \label{version-prop}
For each $\ve>0$, $n \in \N$, and $i \in \{\, 1, \cdots, n \,\}$, there exists a function $v^{\ve}_{i,n}$ satisfying the following$\colon$ 
\begin{enumerate}
\item[{\rm (a)}] $v^{\ve}_{i,n} = \Delta u^{\ve}_{i,n} + \Delta^{-1} V^{\ve}_{i,n}$ a.e. in $\Omega${\rm ;}  
\item[{\rm (b)}] $v^{\ve}_{i,n}$ is upper semicontinuous in $\Omega \setminus \mC^{\ve,-}_{i,n}$. 
On the other hand, $v^{\ve}_{i,n}$ is lower semicontinuous in $\Omega \setminus \mC^{\ve,+}_{i,n}${\rm ;} 
\item[{\rm (c)}] for any $x_0 \in \Omega \setminus \mC^{\ve,-}_{i,n}$ and any sequence of balls 
$B_\rho(x_0) \subset \Omega \setminus \mC^{\ve,-}_{i,n}$, it holds that 
\begin{align*}
\dfrac{1}{\av{B_\rho(x_0)}} \int_{B_\rho(x_0)} v^{\ve}_{i,n} \, dx \downarrow v^{\ve}_{i,n}(x_0) 
\quad \text{as} \quad \rho \downarrow 0. 
\end{align*}
On the other hand, for any $x_1 \in \Omega \setminus \mC^{\ve,+}_{i,n}$ and any sequence of balls 
$B_\rho(x_1) \subset \Omega \setminus \mC^{\ve,+}_{i,n}$, we have 
\begin{align*}
\dfrac{1}{\av{B_\rho(x_1)}} \int_{B_\rho(x_1)} v^{\ve}_{i,n} \, dx \uparrow v^{\ve}_{i,n}(x_1) 
\quad \text{as} \quad \rho \downarrow 0. 
\end{align*}
\end{enumerate}
\end{lem}
%%%%%%%%%%%%%%%%%%%%%%%%%%%%%%%%%%%%%%%%%%%%%%%%%%%%%%
\begin{proof}
Let us set 
\begin{align*}
v^{\ve,\rho}_{i,n}(x) 
= \dfrac{1}{| B_{\rho}(x) |} \int_{B_{\rho}(x)} \left[ \Delta u^{\ve}_{i,n}(y) + \Delta^{-1} V^{\ve}_{i,n}(y)\right] \, dy.  
\end{align*}
If $u^{\ve}_{i,n} \in C^{\infty}(\Omega)$, then Green's formula yields that for each $x_{0} \in \Omega$ 
\begin{align}\label{G-form-1}
\Delta u^{\ve}_{i,n}(x_{0}) + \Delta^{-1} V^{\ve}_{i,n}(x_{0}) 
&= \dfrac{1}{| \pd B_{\rho}(x_{0}) |} \int_{\pd B_{\rho}(x_{0})} \left[ \Delta u^{\ve}_{i,n} + \Delta^{-1} V^{\ve}_{i,n}\right] \, dS \\
& \quad - \int_{B_{\rho}(x_{0})} \left[ \Delta^{2} u^{\ve}_{i,n}(x) + V^{\ve}_{i,n}(x) \right] G_{\rho}(x-x_{0}) \, dx, \notag
\end{align}
where $G_{\rho}$ is Green's function defined by 
\begin{align} \label{G-func}
G_{\rho}(r)= 
\begin{cases}
\vspace{0.1cm}
\dfrac{1}{2} (r-\rho) & \text{if} \quad N=1, \\
\vspace{0.1cm}
\dfrac{1}{2 \pi} \log{\dfrac{\rho}{r}} & \text{if} \quad N=2, \\
\dfrac{1}{N(N-2) \omega(N)} (r^{N-2} - \rho^{N-2}) \quad & \text{if} \quad N \ge 3. \\
\end{cases}
\end{align}
We note that $\omega(N)$ denotes the volume of unit ball in $\R^{N}$. 
Thanks to \eqref{prop-mu-in-ve} and the fact that $G_{\rho'} > G_{\rho}$ if $\rho' > \rho$, we observe from \eqref{G-form-1} that 
\begin{align} \label{v-mono-dec}
v^{\ve,\rho}_{i,n}(x_{0}) \le v^{\ve,\rho'}_{i,n}(x_{0}) 
\quad \text{if} \quad \rho < \rho' \quad \text{and} \quad B_{\rho'}(x_{0}) \subset \Omega \setminus \mC^{\ve,-}_{i,n} 
\end{align}
and 
\begin{align} \label{v-mono-inc}
v^{\ve,\rho}_{i,n}(x_{0}) \ge v^{\ve,\rho'}_{i,n}(x_{0}) 
\quad \text{if} \quad \rho < \rho' \quad \text{and} \quad B_{\rho'}(x_{0}) \subset \Omega \setminus \mC^{\ve,+}_{i,n} 
\end{align}
For general $u^{\ve}_{i,n} \in H^{2}_{0}(\Omega)$, making use of the molification of $\Delta u^{\ve}_{i,n} + \Delta^{-1} V^{\ve}_{i,n}$, 
we are able to verify \eqref{v-mono-dec} and \eqref{v-mono-inc}. 
Hence it follws from \eqref{v-mono-dec} and \eqref{v-mono-inc} that 
\begin{align*}
v^{\ve,\rho}_{i,n}(x) \downarrow \bar{v}^{\ve}_{i,n}(x) 
\quad \text{as} \quad \rho \downarrow 0 \quad \text{in} \quad \Omega \setminus \mC^{\ve,-}_{i,n}
\end{align*}
and 
\begin{align*}
v^{\ve,\rho}_{i,n}(x) \uparrow \tilde{v}^{\ve}_{i,n}(x) 
\quad \text{as} \quad \rho \downarrow 0 \quad \text{in} \quad \Omega \setminus \mC^{\ve,+}_{i,n},
\end{align*}
for some functions $\bar{v}^{\ve}_{i,n}$ and $\tilde{v}^{\ve}_{i,n}$. 

Since $v^{\ve,\rho}_{i,n}$ is continuous in $\Omega$, setting 
\begin{align*}
v^{\ve}_{i,n}(x) = 
\begin{cases}
\bar{v}^{\ve}_{i,n}(x) \quad & \text{if} \quad x \in \Omega \setminus \mC^{\ve,-}_{i,n}, \\
\tilde{v}^{\ve}_{i,n}(x) \quad & \text{if} \quad x \in \Omega \setminus \mC^{\ve,+}_{i,n},
\end{cases}
\end{align*}
we deduce that $v^{\ve}_{i,n}$ is upper semicontinuous in $\Omega \setminus \mC^{\ve,-}_{i,n}$, and 
is lower semicontinuous in $\Omega \setminus \mC^{\ve,+}_{i,n}$. 
Recalling that $\Delta u^{\ve}_{i,n} + \Delta^{-1} V^{\ve}_{i,n} \in L^{2}(\Omega)$, we see that 
\begin{align*}
v^{\ve,\rho}_{i,n} \to \Delta u^{\ve}_{i,n} + \Delta^{-1} V^{\ve}_{i,n} \quad \text{as} \quad \rho \downarrow 0 \quad \text{a.e. in} \quad \Omega. 
\end{align*}
Therefore we conclude that $v^{\ve}_{i,n}= \Delta u^{\ve}_{i,n} + \Delta^{-1} V^{\ve}_{i,n}$ a.e. in $\Omega$. 
\end{proof}

%%%%%%%%%%%%%%%%%%%%%%%%%%%%%%%%%%%%%%%%%%%%%%%%%%%%%%
\begin{lem} \label{version-est}
For any $x_0 \in \mC^{\ve,+}_{i,n}$, it holds that 
\begin{align} \label{vV-ge-f}
v^{\ve}_{i,n}(x_0) - \Delta^{-1} V^{\ve}_{i,n}(x_0) \ge \Delta f(x_0). 
\end{align}
On the other hand, for any $x_1 \in \mC^{\ve,-}_{i,n}$, we have 
\begin{align} \label{vV-le-g}
v^{\ve}_{i,n}(x_1) - \Delta^{-1} V^{\ve}_{i,n}(x_1) \le \Delta g(x_1). 
\end{align}
\end{lem}
%%%%%%%%%%%%%%%%%%%%%%%%%%%%%%%%%%%%%%%%%%%%%%%%%%%%%%
\begin{proof}
Since the proof of \eqref{vV-ge-f} is similar to the proof of Lemma 3.3 in \cite{NO}, we shall prove the latter assertion.  
Let $x_{1} \in \mC^{\ve,-}_{i,n}$. 
Since $\mC^{\ve,+}_{i,n}$ and $\mC^{\ve,-}_{i,n}$ are disjoint, it holds that $\mC^{\ve,-}_{i,n} \subset \Omega \setminus \mC^{\ve,+}_{i,n}$. 
Then there exists a sequence $\{ y_{m} \} \subset \Omega \setminus \mC^{\ve,+}_{i,n}$ with $y_{m} \to x_{1}$ as $m \to \infty$ such that 
\begin{align} \label{u-g-to-0}
u^{\ve}_{i,n}(y_{m}) - g(y_{m}) \uparrow 0 . 
\end{align}
For each $y_{m}$, let $\rho$ be small enough such that 
$B_{\rho,m}:= \{\, y \in \R^{N} \mid | y - y_{m} |< \rho \,\} \subset \Omega \setminus \mC^{\ve,+}_{i,n}$. 
It follows from Green's formula that 
\begin{align}
u^{\ve}_{i,n}(y_{m}) = \dfrac{1}{| \pd B_{\rho,m}| } \int_{\pd B_{\rho,m}} u^{\ve}_{i,n} \, dS 
   - \int_{B_{\rho,m}} \Delta u^{\ve}_{i,n}(y) G_{\rho}(y_{m}-y)\, dy 
\end{align}
and 
\begin{align}\label{g-G-rho}
g(y_{m}) = \dfrac{1}{| \pd B_{\rho,m}| } \int_{\pd B_{\rho,m}} g \, dS 
   - \int_{B_{\rho,m}} \Delta g(y) G_{\rho}(y_{m}-y)\, dy, 
\end{align}
Since $u^{\ve}_{i,n} \le g$ in $\Omega$, we infer from \eqref{u-g-to-0}--\eqref{g-G-rho} that 
\begin{align*}
\liminf_{m \to \infty} \int_{B_{\rho,m}} \left[ \Delta g(y) - \Delta u^{\ve}_{i,n}(y) \right] G_{\rho}(y_{m}-y) \, dy \ge 0. 
\end{align*}
Thanks to Lemma \ref{version-prop}, the relation is reduced to 
\begin{align} \label{gvV-below}
\liminf_{m \to \infty} \int_{B_{\rho,m}} \left[ \Delta g(y) - v^{\ve}_{i,n}(y) + \Delta^{-1} V^{\ve}_{i,n}(y) \right] G_{\rho}(y_{m}-y) \, dy \ge 0. 
\end{align}
Recalling that $V^{\ve}_{i,n} \in H^{2}_{0}(\Omega)$, we observe from the elliptic regularity, e.g., see \cite{GT}, 
that $\Delta^{-1}V^{\ve}_{i,n} \in H^{4}(\Omega)$. 
We note that Sobolev's embedding theorem implies that $\Delta^{-1} V^{\ve}_{i,n}$ is continuous in $\Omega$ provided $N \le 7$. 
Since $v^{\ve}_{i,n}$ is lower semicontinuous in $\Omega \setminus \mC^{\ve,+}_{i,n}$, 
there exists a point $y_{m,\rho} \in \con{B}_{\rho,m}$ such that the maxmum of 
$\Delta g(y) - v^{\ve}_{i,n}(y) + \Delta^{-1} V^{\ve}_{i,n}(y)$ in $\con{B}_{\rho,m}$ attains at $y=y_{m,\rho}$. 
Hence it follows from \eqref{gvV-below} that there exists a sequence $\{ \vd_{m}\}$ with $\vd_{m} \downarrow 0$ as $m \to \infty$ such that 
\begin{align*}
\Delta g(y_{m,\rho}) - v^{\ve}_{i,n}(y_{m,\rho}) + \Delta^{-1} V^{\ve}_{i,n}(y_{m,\rho}) \ge - \vd_{m}. 
\end{align*}
As $m \to \infty$, $y_{m,\rho}$ converges to a point $y_{\rho} \in \{\, y \in \R^{N} \mid | y - x_{1}| \le \rho \,\}$ up to a subsequence, 
for the sequence $\{ y_{m,\rho} \}$ is bounded. 
Thanks to the lower semicontinuity of $v^{\ve}_{i,n}$, we find 
\begin{align*}
\Delta g(y_{\rho}) - v^{\ve}_{i,n}(y_{\rho}) + \Delta^{-1} V^{\ve}_{i,n}(y_{\rho}) \ge 0 
\end{align*}
for any $\rho>0$ small enough. 
Letting $\rho \downarrow 0$ and making use of the lower semicontinuity of $v^{\ve}_{i,n}$, we conclude \eqref{vV-le-g}.  
\end{proof}

%%%%%%%%%%%%%%%%%%%%%%%%%%%%%%%%%%%%%%%%%%%%%%%%%%%%%%
\begin{lem} \label{Lap-bdd}
For each $\ve>0$, $n \in \N$, and $i=1, \ldots, n$, it holds that $\Delta u^{\ve}_{i,n} \in L^{\infty}(\Omega)$. 
Moreover, there exists a positive constant $C$ independent of $\ve$, $n$, and $i$, such that 
\begin{align} \label{pre-Li-Delta-u}
\| \Delta u^{\ve}_{i,n} \|_{L^{\infty}(\Omega)}  
&\le C E(u_0)^{\frac{1}{2}} + \| V^{\ve}_{i,n} \|_{L^2(\Omega)} 
        + \max\{ \| \Delta f_\ve \|_{L^\infty(\Omega)}, \| \Delta g \|_{L^\infty(\Omega)} \} \\
& \qquad + C \left( \dfrac{E(u^{\ve}_{i-1,n}) - E(u^{\ve}_{i,n})}{\tau_n} \right)^{\frac{1}{2}}. \notag
\end{align}
\end{lem}
%%%%%%%%%%%%%%%%%%%%%%%%%%%%%%%%%%%%%%%%%%%%%%%%%%%%%%
\begin{proof}
Let us set 
\begin{align} \label{def-U}
U^{\ve}_{i,n}:= u^{\ve}_{i,n} + (\Delta^2)^{-1} V^{\ve}_{i,n}, 
\end{align}
where $(\Delta^2)^{-1} V^{\ve}_{i,n}$ denotes the unique solution of 
\begin{align*}
\begin{cases}
\Delta^2 w = V^{\ve}_{i,n} \quad & \text{in} \quad \Omega, \\
w=0, \, \Delta w =0, \quad & \text{on} \quad \pd \Omega. 
\end{cases}
\end{align*}
Fix $x_0 \in \Omega$ arbitrarily. Let $B_\rho$ denote the ball center $x_0$ and the radius $\rho$. 
For any $R>0$ with $\con{B}_R \subset \Omega$, let $\vz \in C^{\infty}_c(B_R)$ be a test function with 
$\vz=1$ in $B_{2R/3}$, $0 \le \vz \le 1$ elsewhere. 
By the same argument as in the proof of Lemma 3.4 in \cite{NO}, we see that for any $x \in B_{R/2}$ 
\begin{align} \label{u-G-form}
v^{\ve}_{i,n}(x) = - \int_{B_{R/2}} G_R(x-y) d \mu^{\ve}_{i,n}(y) 
  - I_1(x) + \va(x) 
\end{align}
with 
\begin{align*}
I_1(x) := \int_{D_{R/2}} \vz(y) G_R(x-y) \Delta^2 U^{\ve}_{i,n}(y) \, dy. 
\end{align*}
and 
\begin{align*}
\av{\va(x)} \le C_1 \| \Delta U^{\ve}_{i,n} \|_{L^2(\Omega)} \quad \text{in} \quad B_{R/2}. 
\end{align*}
Here $G_R$ is Green's function given by \eqref{G-func} with $\rho=R$. 
We note that for any $x \in B_{R/3}$ 
\begin{align*}
\av{I_1(x)} & \le C \av{\mu^{\ve}_{i,n}}(D_{R/2}) 
 \le C_2 E(u_0)^{\frac{1}{2}} + C_3 \left( \dfrac{E(u^{\ve}_{i-1,n}) - E(u^{\ve}_{i,n})}{\tau_n} \right)^{\frac{1}{2}}, 
\end{align*}
where the constants $C_2$ and $C_3$ are independent of $\ve$ and $n$. 
Set 
\begin{align*}
\tilde{G}_{R}(x) = \int_{B_{R/2}} G_R(x-y) d \mu^{\ve}_{i,n}(y). 
\end{align*}
Thanks to Lemma \ref{version-est}, we observe from \eqref{u-G-form} that  
\begin{align*}
\tilde{G}_{R}(x) &= -v^{\ve}_{i,n}(x) -I_1(x) + \va(x) 
 \le -\Delta^{-1}V^{\ve}_{i,n}(x) - \Delta f_\ve(x) + \av{I_1(x)} + \va(x) \\
& <  C_1 \| \Delta U^{\ve}_{i,n} \|_{L^2(\Omega)} 
     + C_2 E(u_0)^{\frac{1}{2}} + C_3 \left( \dfrac{E(u^{\ve}_{i-1,n}) - E(u^{\ve}_{i,n})}{\tau_n} \right)^{\frac{1}{2}}  \\
& \qquad   + C_4 \| V^{\ve}_{i,n} \|_{L^2(\Omega)} + \| \Delta f_\ve \|_{L^\infty(\Omega)} 
\qquad \text{in} \quad \mC^{\ve,+}_{i,n} \cap B_{R/3},  
\end{align*}
and while 
\begin{align*}
\tilde{G}_{R}(x) &= -v^{\ve}_{i,n}(x) - I_1(x) + \va(x) 
 \ge -\Delta^{-1}V^{\ve}_{i,n}(x) - \Delta g(x) -\av{I_1(x)} + \va(x) \\
& > - C_1 \| \Delta U^{\ve}_{i,n} \|_{L^2(\Omega)} 
    - C_2 E(u_0)^{\frac{1}{2}} - C_3 \left( \dfrac{E(u^{\ve}_{i-1,n}) - E(u^{\ve}_{i,n})}{\tau_n} \right)^{\frac{1}{2}}  \\
& \qquad - C_4 \| V^{\ve}_{i,n} \|_{L^2(\Omega)} - \| \Delta g \|_{L^\infty(\Omega)}
 \quad\,\,\, \text{in} \quad \mC^{\ve,-}_{i,n} \cap B_{R/3}. 
\end{align*}
Then, along the same lines as in the proof of Theorems 1.6 and 1.10 of \cite{La}, we deduce that 
\begin{align*}
\limsup_{d(x, \mC^{\ve,+}_{i,n}) \to 0} \tilde{G}_{R}(x) 
& \le C_1 \| \Delta U^{\ve}_{i,n} \|_{L^2(\Omega)} 
   + C_2 E(u_0)^{\frac{1}{2}} + C_3 \left( \dfrac{E(u^{\ve}_{i-1,n}) - E(u^{\ve}_{i,n})}{\tau_n} \right)^{\frac{1}{2}} \\
& \qquad + C_4 \| V^{\ve}_{i,n} \|_{L^2(\Omega)} + \| \Delta f_\ve \|_{L^\infty(\Omega)} 
\end{align*}
and 
\begin{align*} 
\limsup_{d(x, \mC^{\ve,-}_{i,n}) \to 0} \tilde{G}_{R}(x) 
& \ge - C_1 \| \Delta U^{\ve}_{i,n} \|_{L^2(\Omega)} 
    - C_2 E(u_0)^{\frac{1}{2}} - C_3 \left( \dfrac{E(u^{\ve}_{i-1,n}) - E(u^{\ve}_{i,n})}{\tau_n} \right)^{\frac{1}{2}} \\ 
& \qquad - C_4 \| V^{\ve}_{i,n} \|_{L^2(\Omega)} - \| \Delta g \|_{L^\infty(\Omega)}. 
\end{align*}
Thus the maximal principle implies that 
\begin{align*}
| \tilde{G}_{R}(x) |  
& \le C_1 \| \Delta U^{\ve}_{i,n} \|_{L^2(\Omega)} 
+ C_2 E(u_0)^{\frac{1}{2}} + C_3 \left( \dfrac{E(u^{\ve}_{i-1,n}) - E(u^{\ve}_{i,n})}{\tau_n} \right)^{\frac{1}{2}}  \\
& \qquad + C_4 \| V^{\ve}_{i,n} \|_{L^2(\Omega)} 
        + \max\{\, \| \Delta f_\ve \|_{L^\infty(\Omega)}, \| \Delta g \|_{L^\infty(\Omega)} \,\}
 \quad \text{in} \quad B_{R/3}.  
\end{align*}
Combining \eqref{u-G-form} with Theorem \ref{Radon-1} and Lemma \ref{version-prop}, we obtain \eqref{pre-Li-Delta-u}.  
\end{proof}

%%%%%%%%%%%%%%%%%%%%%%%%%%%%%%%%%%%%%%
\begin{lem}{{\rm (\cite{CFT})}} \label{CFT-lem}
Let $N \le 3$. Let $w \in H^2(\Omega)$ be a non-negative function satisfying 
\begin{align*}
\| \Delta w \|_{L^\infty(\Omega)} \le M_0. 
\end{align*}
Then there exists a constant $M$ depending only on $M_0$ such that if 
\begin{align*}
x_0 \in J := \{ x \in \Omega \mid w(x)=0 \}
\end{align*}
then it holds that 
\begin{align}
| w(x) | \le M | x-x_0 |^2, \quad | \nabla w(x) | \le M | x - x_0 |, \quad \text{in} \quad B(x_0, \rho/2), 
\end{align}
where $\rho= {\rm dist}(x_0, \pd \Omega)$. 
\end{lem}
%%%%%%%%%%%%%%%%%%%%%%%%%%%%%%%%%%%%%%

%%%%%%%%%%%%%%%%%%%%%%%%%%%%%%%%%%%%%%
\begin{lem} \label{D2-bdd-lem-1}
For any $x \in \Omega \setminus (\mC^{\ve,+}_{i,n} \cup \mC^{\ve,-}_{i,n})$, it holds that 
\begin{align*}
| D^{2} u^{\ve}_{i,n}(x) | \le C ( \| \Delta u^{\ve}_{i,n} \|_{L^{2}(\Omega)} + \| V^{\ve}_{i,n} \|_{L^{2}(\Omega)} 
  & + \| D^{2} f \|_{L^{\infty}(\Omega)} + \| \Delta^{2} f \|_{L^{2}(\Omega)} \\ 
  & + \| D^{2} g \|_{L^{\infty}(\Omega)} + \| \Delta^{2} g \|_{L^{2}(\Omega)}). \notag
\end{align*}
\end{lem}
%%%%%%%%%%%%%%%%%%%%%%%%%%%%%%%%%%%%%%
\begin{proof}
Since $u^{\ve}_{i,n}$ is continuous in $\Omega$, we see that $\vd:={\rm dist}(\mC^{\ve,+}_{i,n} \cup \mC^{\ve,-}_{i,n}, \pd \Omega) > 0$. 
To begin with, recall that 
\begin{align*}
\Delta^{2} u^{\ve}_{i,n} + V^{\ve}_{i,n} =0 \quad \text{in} \quad \Omega \setminus \Omega_{\vd}, 
\end{align*}
where $\Omega_{\vd}= \{ x \in \Omega \mid {\rm dist}(x, \pd \Omega) > \vd \}$. 
By the elliptic regularity theory (i.e., see \cite{GT}), we deduce from $\pd \Omega \in C^{4}$ that 
\begin{align} \label{Delta-H2-in}
\| \Delta u^{\ve}_{i,n} \|_{H^{2}(\Omega \setminus \Omega_{\rho})} 
\le C (\| \Delta u^{\ve}_{i,n} \|_{L^{2}(\Omega)} + \| V^{\ve}_{i,n} \|_{L^{2}(\Omega)})
\quad \text{for any} \quad 0 < \rho < \vd, 
\end{align}
where the constant $C>0$ is independent of $i$, $n$ and $\ve$. 
Setting $\tilde{u}:= \eta u^{\ve}_{i,n}$, where $\eta \in C^{\infty}_{c}(\Omega \setminus \Omega_{\vd})$ with $0 \le \eta \le 1$ and 
\begin{align*}
\eta(x)= 
\begin{cases}
1 \quad & \text{in} \quad \Omega \setminus \Omega_{3\vd/4}, \\
0 \quad & \text{in} \quad \Omega_{7\vd/8}, 
\end{cases}
\end{align*}
we find 
\begin{align*}
\begin{cases}
\Delta^{2} \tilde{u} = F(\eta,u^{\ve}_{i,n}) - \eta V^{\ve}_{i,n} \quad & \text{in} \quad \Omega \setminus \Omega_{7\vd/8}, \\
\tilde{u}= \pd_{\nu} \tilde{u}=0 \quad & \text{on} \quad \pd (\Omega \setminus \Omega_{7\vd/8}), 
\end{cases}
\end{align*}
where 
\begin{align*}
F(\eta,u^{\ve}_{i,n}):= \Delta^{2} \eta u^{\ve}_{i,n} + 2 \nabla \Delta \eta \cdot \nabla u^{\ve}_{i,n} 
 + 2 \Delta (\nabla \eta \cdot \nabla u^{\ve}_{i,n}) + 2 \Delta \eta \Delta u^{\ve}_{i,n} 
 + 2 \nabla \eta \cdot \nabla \Delta u^{\ve}_{i,n}.
\end{align*}
Thanks to Theorem 2.20 in \cite{GGS}, we observe from \eqref{Delta-H2-in} and $\pd \Omega \in C^{4}$ that  
\begin{align*}
\| \tilde{u} \|_{H^{4}(\Omega \setminus \Omega_{7 \vd/8})} 
& \le C (\|F(\eta,u^{\ve}_{i,n}) \|_{L^{2}(\Omega \setminus \Omega_{7\vd/8})} 
          + \|\eta V^{\ve}_{i,n} \|_{L^{2}(\Omega \setminus \Omega_{7\vd/8})}) \\
& \le C (\| \Delta u^{\ve}_{i,n} \|_{L^{2}(\Omega)} + \| V^{\ve}_{i,n} \|_{L^{2}(\Omega)}).           
\end{align*}
Since $\| u^{\ve}_{i,n} \|_{H^{4}(\Omega \setminus \Omega_{3\vd/4})} = \| \tilde{u} \|_{H^{4}(\Omega \setminus \Omega_{3\vd/4})} 
\le \| \tilde{u} \|_{H^{4}(\Omega \setminus \Omega_{7 \vd/8})}$, the estimate implies 
\begin{align*}
\| D^{2} u^{\ve}_{i,n} \|_{H^{2}(\Omega \setminus \Omega_{3\vd/4})} 
 \le C (\| \Delta u^{\ve}_{i,n} \|_{L^{2}(\Omega)} + \| V^{\ve}_{i,n} \|_{L^{2}(\Omega)}). 
\end{align*}
Then it follows from Sobolev's embedding theorem that 
\begin{align} \label{D2-uni-est-1}
\| D^{2} u^{\ve}_{i,n} \|_{L^{\infty}(\Omega \setminus \Omega_{3 \vd/4})} 
 \le C ( \| \Delta u^{\ve}_{i,n} \|_{L^{2}(\Omega)} + \| V^{\ve}_{i,n} \|_{L^{2}(\Omega)}), 
\end{align}
where the constant $C$ is independent of $i$, $n$, and $\ve$. 

Let $x_{0} \in \Omega_{\vd/2} \setminus (\mC^{\ve,+}_{i,n} \cup \mC^{\ve,-}_{i,n})$ satisfy 
${\rm dist}(x_{0},\mC^{\ve,+}_{i,n} \cup \mC^{\ve,-}_{i,n}) \le \vd$. 
Here we may assume that 
\begin{align*}
{\rm dist}(x_{0},\mC^{\ve,+}_{i,n} \cup \mC^{\ve,-}_{i,n})= {\rm dist}(x_{0},\mC^{\ve,-}_{i,n}). 
\end{align*}
From Lemmas \ref{Lap-bdd} and \ref{CFT-lem}, there exists a constant $C>0$ independent of $i$, $n$, and $\ve$ such that 
\begin{gather}
| (u^{\ve}_{i,n} - g)(x)| 
 \le C \| \Delta (u^{\ve}_{i,n} - g) \|_{L^{\infty}(\Omega)} {\rm dist}(x,\mC^{\ve,-}_{i,n})^{2}, \label{est-1-u-g}\\
| \nabla (u^{\ve}_{i,n} - g)(x)| 
  \le C \| \Delta (u^{\ve}_{i,n} - g) \|_{L^{\infty}(\Omega)} {\rm dist}(x,\mC^{\ve,-}_{i,n}), \label{est-1-grad-u-g}
\end{gather}
in $B(x_{0},d)$, where $d= {\rm dist}(x_{0},\mC^{\ve,-}_{i,n})$. 
We consider 
\begin{align*}
w_{d}(x)= \dfrac{1}{d^{2}}(u^{\ve}_{i,n}-g)(d(x-x_{0})) \quad \text{in} \quad B(x_{0},1). 
\end{align*}
For the simplicity, we may assume $x_{0}=0$. Then it follows from \eqref{est-1-u-g}-\eqref{est-1-grad-u-g} that 
\begin{align*}
| w_{d}(x) | \le C \| \Delta (u^{\ve}_{i,n} - g) \|_{L^{\infty}(\Omega)}, \,\,\,
| \nabla w_{d}(x) | \le C \| \Delta (u^{\ve}_{i,n} - g) \|_{L^{\infty}(\Omega)}, \quad \text{in} \quad B(0,1). 
\end{align*}
Since  
\begin{align*}
\Delta^2 w_d(x) = - d^2 V^{\ve}_{i,n}(d(x-x_{0})) - d^2 \Delta^2 g(d(x-x_{0})) \quad \text{in} \quad B(0,1), 
\end{align*}
we observe from the same argument as in the derivation of \eqref{D2-uni-est-1} that 
\begin{align*}
| D^2 w_d(x) | \le C ( \| \Delta u^{\ve}_{i,n} \|_{L^{2}(\Omega)} + \| V^{\ve}_{i,n} \|_{L^{2}(\Omega)} 
                      + \sum^{2}_{i=1} \| \Delta^{i} g \|_{L^{2}(\Omega)} ) 
                       \quad \text{in} \quad B(0, \tfrac{1}{2}). 
\end{align*}
Thus it holds that 
\begin{align} \label{D2-uni-est-2}
| D^2 u^{\ve}_{i,n}(x)| \le C ( \| \Delta u^{\ve}_{i,n} \|_{L^{2}(\Omega)} & + \| V^{\ve}_{i,n} \|_{L^{2}(\Omega)} \\
  & + \| D^{2} g \|_{L^{\infty}(\Omega)} + \| \Delta^{2} g \|_{L^{2}(\Omega)}) \notag
\quad \text{in} \quad B(x_{0}, d/2). 
\end{align}
If ${\rm dist}(x_{0},\mC^{\ve,+}_{i,n} \cup \mC^{\ve,-}_{i,n})= {\rm dist}(x_{0},\mC^{\ve,+}_{i,n})$, 
then we obtain \eqref{D2-uni-est-2} replaced $g$ by $f$. 
We thus completed the proof. 
\end{proof}

%%%%%%%%%%%%%%%%%%%%%%%%%%%%%%%%%%%%%%
\begin{thm} \label{C-1-1_u-in}
It holds that $u^{\ve}_{i,n} \in W^{2, \infty}(\Omega)$. Moreover, there exists a positive constant $C$ independent of $\ve$ 
and $n$ such that 
\begin{align*}
\tau_n \sum^{n}_{i=1} \| D^2 u^{\ve}_{i,n} \|^2_{L^\infty(\Omega)} 
\le C ( E(u_0) & + \| D^{2} f \|_{L^{\infty}(\Omega)} + \| \Delta^{2} f \|_{L^{2}(\Omega)} \\
      & + \| D^{2} g \|_{L^{\infty}(\Omega)} + \| \Delta^{2} g \|_{L^{2}(\Omega)}). 
\end{align*}
\end{thm}
%%%%%%%%%%%%%%%%%%%%%%%%%%%%%%%%%%%%%%
\begin{proof}
Let $e_j$ be the unit vector in the direction of the positive $x_j$ axis. 
Fix $x \in \Omega$. For $|h| \in \R$ small enough, we consider the second order differencial quotient  
\begin{align*}
D^{2}_{h} u^{\ve}_{i,n}(x) = \dfrac{u^{\ve}_{i,n}(x + h e_j) + u^{\ve}_{i,n}(x - h e_j) - 2 u^{\ve}_{i,n}(x)}{2 h^{2}}. 
\end{align*}
If ${\rm dist}(x, \mC^{\ve,+}_{i,n} \cup \mC^{\ve,-}_{i,n}) < 4 |h|$, 
then there exists $x_0 \in \mC^{\ve,+}_{i,n} \cup \mC^{\ve,-}_{i,n}$ such that 
\begin{align*}
|x - x_{0} | =  {\rm dist}(x, \mC^{\ve,+}_{i,n} \cup \mC^{\ve,-}_{i,n}) < 4 |h|. 
\end{align*}
We may assume $x_0 \in \mC^{\ve,-}_{i,n}$
Making use of \eqref{est-1-u-g}, we find 
\begin{align*}
& | D^{2}_{h} (u^{\ve}_{i,n} - g)(x)| \\
& \,\,\, \le \dfrac{C}{h^2} \| \Delta (u^{\ve}_{i,n} - g) \|_{L^{\infty}(\Omega)} 
 \left[ {\rm dist}(x+h e_{j},\mC^{\ve,-}_{i,n})^{2} + {\rm dist}(x-h e_{j},\mC^{\ve,-}_{i,n})^{2} 
         + {\rm dist}(x,\mC^{\ve,-}_{i,n})^{2} \right]\\
& \,\,\, \le C \| \Delta (u^{\ve}_{i,n} - g) \|_{L^{\infty}(\Omega)}. 
\end{align*}
On the other hand, if ${\rm dist}(x, \mC^{\ve,+}_{i,n} \cup \mC^{\ve,-}_{i,n}) \ge 4 |h|$, 
then we observe from Lemma \ref{D2-bdd-lem-1} that 
\begin{align*}
| D^{2}_{h} u^{\ve}_{i,n}(x) | \le | D_{x_{j} x_{j}} u^{\ve}_{i,n}(\tilde{x}) |
 \le C ( & \| \Delta u^{\ve}_{i,n} \|_{L^{2}(\Omega)} + \| V^{\ve}_{i,n} \|_{L^{2}(\Omega)} + \| D^{2} f \|_{L^{\infty}(\Omega)}\\ 
  &  + \| \Delta^{2} f \|_{L^{2}(\Omega)} + \| D^{2} g \|_{L^{\infty}(\Omega)} + \| \Delta^{2} g \|_{L^{2}(\Omega)}), 
\end{align*}
where $\tilde{x} \in B(x, 2 {\rm dist}(x, \mC^{\ve,+}_{i,n} \cup \mC^{\ve,-}_{i,n}))$. 
Consequently we see that, for any $x \in \Omega$, if $|h|$ is small enough, 
\begin{align*}
| D^{2}_{h} u^{\ve}_{i,n} (x) | 
\le C ( & \| \Delta u^{\ve}_{i,n} \|_{L^{\infty}(\Omega)} + \| V^{\ve}_{i,n} \|_{L^{2}(\Omega)} + \| D^{2} f \|_{L^{\infty}(\Omega)}\\ 
  &  + \| \Delta^{2} f \|_{L^{2}(\Omega)} + \| D^{2} g \|_{L^{\infty}(\Omega)} + \| \Delta^{2} g \|_{L^{2}(\Omega)}),  
\end{align*} 
where $C>0$ is independent of $x$ and $h$. 
Therefore we deduce that 
\begin{align} \label{pre-uni-D2}
| D_{x_{j} x_{j}} u^{\ve}_{i,n} (x) | 
\le C ( & \| \Delta u^{\ve}_{i,n} \|_{L^{\infty}(\Omega)} + \| V^{\ve}_{i,n} \|_{L^{2}(\Omega)} + \| D^{2} f \|_{L^{\infty}(\Omega)}\\ 
  &  + \| \Delta^{2} f \|_{L^{2}(\Omega)} + \| D^{2} g \|_{L^{\infty}(\Omega)} + \| \Delta^{2} g \|_{L^{2}(\Omega)}) 
  \quad \text{in} \quad \Omega.  \notag
\end{align} 
Combining \eqref{pre-uni-D2} with Proposition \ref{est-mini} and Lemma \ref{Lap-bdd}, we obtain the conclusion. 
\end{proof}

Let us set  
\begin{align}
\mC^{+}_{i,n} &= \{\, x \in \Omega \setminus \Omega_0 \mid u_{i,n}(x)= f(x) \,\}, \\
\mC^{-}_{i,n} &= \{\, x \in \Omega \setminus \Omega_0 \mid u_{i,n}(x)= g(x) \,\},  
\end{align}
where $\Omega_0$ is defined in \eqref{Omega_0}. 
%Remark that $\mC^{+}_{i,n}$ and $\mC^{-}_{i,n}$ are disjoint, for $u_{i,n}$ is continuous.  
%%%%%%%%%%%%%%%%%%%%%%%%%%%%%%%%%%%%%%
\begin{thm} \label{Radon-2}
As $\ve \downarrow 0$, the signed measure $\mu^{\ve}_{i,n}$ converges to a signed Radon measure $\mu_{i,n}$ in $\Omega$ 
defined by 
\begin{align*}
\mu_{i,n} = 
\begin{cases}
\Delta^2 u_{i,n} + V_{i,n} \quad & \text{in} \quad \Omega \setminus \Omega_0, \\
\Delta^2 f \quad & \text{in} \quad \Omega_0. 
\end{cases}
\end{align*} 
Moreover it holds that  
${\rm supp}\, \mu_{i,n} \subset \mC^{+}_{i,n} \cup \mC^{-}_{i,n} \cup \Omega_0$,  
\begin{align*}
\mu_{i,n} 
\begin{cases}
\ge 0 \quad & \text{in} \quad \mC^{+}_{i,n}, \\
\le 0 \quad & \text{in} \quad \mC^{-}_{i,n},  
\end{cases}
\end{align*}
and there exists a positive constant $C>0$ independent of $n$ such that 
\begin{align} \label{bdd-meas-2}
\tau_n \sum^{n}_{i=1} \mu_{i,n}(\Omega)^2 < C E(u_0) + T \| \Delta^2 f \|^2_{L^{\infty}(\Omega_0)}. 
\end{align}
\end{thm}
%%%%%%%%%%%%%%%%%%%%%%%%%%%%%%%%%%%%%%
\begin{proof} 
To begin with, we shall prove that $\mu^{\ve}_{i,n} \lfloor_{\Omega_0} \rightharpoonup \Delta^2 f$ as $\ve \downarrow 0$, i.e., 
\begin{align} \label{Ome_0-0}
\int_\Omega \left[ \Delta u^{\ve}_{i,n} \Delta \vp + V^{\ve}_{i,n} \vp \right] \, dx 
\to \int_\Omega \Delta f \Delta \vp \, dx \quad \text{as} \quad \ve \downarrow 0 
\quad \text{for any} \quad \vp \in C^{\infty}_{c}(\Omega_0) . 
\end{align}
Since it holds that 
\begin{align*}
\av{u^{\ve}_{i,n}(x) - f(x)} \le \ve \quad \text{in} \quad \Omega_0, 
\end{align*}
we infer that 
\begin{align} \label{Ome_0-1}
\av{\int_\Omega \left( \Delta u^{\ve}_{i,n} - \Delta f \right)\Delta \vp \, dx } 
 \le \| u^{\ve}_{i,n} - f \|_{L^\infty(\Omega_0)} \int_\Omega \av{\Delta^2 \vp} \, dx 
 \le \ve \int_\Omega \av{\Delta^2 \vp} \, dx. 
\end{align}
On the other hand, from 
\begin{align*}
\av{V^{\ve}_{i,n}} 
 \le \dfrac{1}{\tau_n} \left\{ \av{u^{\ve}_{i,n} - f} + \av{u^{\ve}_{i-1,n} - f} \right\} \le \dfrac{2}{\tau_n} \ve, 
\end{align*}
we have 
\begin{align} \label{Ome_0-2}
\av{\int_\Omega V^{\ve}_{i,n} \vp \, dx} 
 \le \dfrac{2}{\tau_n} \ve \int_\Omega \av{\vp}\, dx. 
\end{align}
Then \eqref{Ome_0-1} and \eqref{Ome_0-2} implies \eqref{Ome_0-0}. 

From now on, we write $\mu^{\ve}_{i,n} \lfloor_{\Omega \setminus \Omega_0} = \nu^{\ve,+}_{i,n} - \nu^{\ve,-}_{i,n}$, where 
$\nu^{\ve,\pm}_{i,n}$ are positive measure in $\Omega$ 
with ${\rm supp}\, \nu^{\ve,\pm}_{i,n} \subset \mC^{\ve,\pm}_{i,n}$, respectively. 
By the proof of Theorem \ref{Radon-1}, there exist measures $\bar{\mu}^{\pm}_{i,n}$ in $\Omega$ such that 
\begin{align*}
\nu^{\ve,\pm}_{i,n} \rightharpoonup \bar{\mu}^{\pm}_{i,n} \quad \text{as} \quad \ve \downarrow 0, 
\end{align*} 
i.e., 
\begin{align*}
\int_\Omega \vz d \nu^{\ve,\pm}_{i,n} \to \int_\Omega \vz d \bar{\mu}^{\pm}_{i,n} 
\quad \text{for any} \quad \vz \in C_c(\Omega \setminus \Omega_0) \quad \text{as} \quad \ve \downarrow 0. 
\end{align*}
Since 
\begin{align*}
\int_\Omega \vz d \nu^{\ve,\pm}_{i,n} 
 = \pm \int_\Omega \left[ \Delta u^{\ve}_{i,n} \Delta \vz + V^{\ve}_{i,n} \vz \right] \, dx 
 \to \pm \int_\Omega \left[ \Delta u_{i,n} \Delta \vz + V_{i,n} \vz \right] \, dx 
\end{align*}
for any $\vz \in C^2_c(\Omega \setminus \Omega_0)$ as $\ve \downarrow 0$,
it holds that $\bar{\mu}^{\pm}_{i,n} = \pm(\Delta^2 u_{i,n} + V_{i,n})$. We claim that 
\begin{align} \label{supp-mu-in-2}
{\rm supp} \, \bar{\mu}^+_{i,n} \subset \mC^{+}_{i,n}, \qquad  {\rm supp} \, \bar{\mu}^{-}_{i,n} \subset \mC^{-}_{i,n}. 
\end{align}
It is sufficient to show the former relation. 
Let $x_0 \in \Omega \setminus (\mC^{+}_{i,n} \cup \Omega_0)$. 
Then there exist a neighborhood $W \subset \Omega \setminus \Omega_0$ of $x_0$ and a constant $\vd>0$ such that 
\begin{align*}
u_{i,n}(x) - f(x) > \vd \quad \text{in} \quad W. 
\end{align*}
Since $u^{\ve}_{i,n}$ uniformly converges to $u_{i,n}$, there exists $\ve_*>0$ such that for any $\ve < \ve_*$ 
\begin{align*}
\av{u^{\ve}_{i,n}(x) - u_{i,n}(x)} < \dfrac{\vd}{3} \quad \text{in} \quad W.  
\end{align*}
Thus, for any $\ve < \min\{ \ve_*, \vd/3\}$, we have 
\begin{align*}
u^{\ve}_{i,n}(x) - f_\ve(x) 
 > u_{i,n}(x) - f(x) - \av{u^{\ve}_{i,n}(x) - u_{i,n}(x)} - \av{f_\ve(x) - f(x)} 
 > \dfrac{\vd}{3} \quad \text{in} \quad W, 
\end{align*}
i.e., $W \subset \Omega \setminus (\mC^{\ve,+}_{i,n} \cup \Omega_0)$ for $\ve>0$ small enough. 
Hence we infer that  for any $\vz \in C_c(W)$ 
\begin{align*}
\int_\Omega \vz d \bar{\mu}^{+}_{i,n} 
 = \lim_{\ve \downarrow 0} \int_\Omega \vz d \nu^{\ve,+}_{i,n} =0. 
\end{align*}
Therefore the relation \eqref{supp-mu-in-2} holds. 

Finally we turn to \eqref{bdd-meas-2}. 
It follows from the proof of Theorem \ref{Radon-1} that 
\begin{align*}
\bar{\mu}^{\pm}_{i,n}(\Omega) \le \liminf_{\ve \downarrow 0} \nu^{\ve,\pm}_{i,n}(\Omega) 
\le C(U) E(u_0)^{\frac{1}{2}} + C(U) \left( \dfrac{ E(u^{\ve}_{i-1,n}) - E(u^{\ve}_{i,n}) }{\tau_n} \right)^{\frac{1}{2}}.  
\end{align*}
Moreover it holds that  
\begin{align*}
\tau_{n} \sum^{n}_{i=1} \mu_{i,n}(\Omega_{0})^{2} \le C T \| \Delta^2 f \|^2_{L^{\infty}(\Omega_0)},  
\end{align*}
where the constant $C$ is independent of $n$. 
Recalling that ${\rm sup}\, \mu_{i,n} \subset \mC^{+}_{i,n} \cup \mC^{-}_{i,n} \cup \Omega_0$, we obtain 
\begin{align*}
\tau_n \sum^{n}_{i=1} \mu_{i,n}(\Omega)^2 
 &\le C_1 T E(u_0) + C_2 \sum^{n}_{i=1} \{ E(u^{\ve}_{i-1,n}) - E(u^{\ve}_{i,n}) \} + T \| \Delta^2 f \|^2_{L^{\infty}(\Omega_0)} \\
 &\le C E(u_0) + T \| \Delta^2 f \|^2_{L^{\infty}(\Omega_0)}. 
\end{align*}
We thus completed the proof. 
\end{proof}

%%%%%%%%%%%%%%%%%%%%%%%%%%%%%%%%%%%%%%%%%%%%%%%%%%%%%%
%%%%%%%%%%%%%%%%%%%%%%%%%%%%%%%%%%%%%%%%%%%%%%%%%%%%%%
%%%%%%%%%%%%%%%%%%%%%%%%%%%%%%%%%%%%%%%%%%%%%%%%%%%%%%

\section{Proof of the main theorem} \label{proof-of-main-theorem}
In this section, we prove Theorem \ref{main-thm}. 
First we shall prove the convergence of the piecewise linear interpolation $u_{n}$ of $\{ u_{i,n} \}$. 
The proof is followed from the uniform estimates on $\{u_{n} \}$. 
Since the estimates have already obtained by Proposition \ref{est-mini}, we are able to prove the following result along the 
same lines as in the proof of Theorem 4.1 in \cite{NO}. 

%%%%%%%%%%%%%%%%%%%%%%%%%%%%%%%
\begin{thm} \label{conv-1}
Let $u_n$ be the piecewise linear interpolation of $\{ u_{i,n} \}$. Then there exists a function 
\begin{align*}
u \in L^\infty(0,T;H^2_0(\Omega)) \cap H^1(0,T;L^2(\Omega)) 
\end{align*}
such that for any $T<\infty$ 
\begin{align*}
u_n \rightharpoonup u \quad \text{in} \quad L^2(0,T;H^2(\Omega)) \cap H^1(0,T;L^2(\Omega)) 
\quad \text{as} \quad n \to \infty, 
\end{align*}
up to a subsequence. Moreover 
\begin{align*}
\int^T_0 \!\!\! \int_\Omega \av{\pd_t u}^2 \, dx dt \le 2 E(u_0), 
\end{align*}
$f(x) \le u(x,t) \le g(x)$ for $x \in \Omega$ and every $t \in [\, 0,T \,]$, and for each $\va \in (\, 0,1/2 \,)$, 
it holds that 
\begin{align*}
u_n \to u  \quad \text{in} \quad C^{0,\va}([\, 0,T \,];L^2(\Omega)) \quad \text{as} \quad n \to \infty. 
\end{align*} 
\end{thm}
%%%%%%%%%%%%%%%%%%%%%%%%%%%%%%%%

Next we investigate the regularity of the limit $u$ obtained by Theorem \ref{conv-1}. 
The proof depends only on the uniform estimate on $u_{n}$ obtained in Theorem \ref{C-1-1_u-in}. 
The same argument as in the proof of Theorems 4.2 and 4.3 in \cite{NO} gives us the following:
%%%%%%%%%%%%%%%%%%%%%%%%%%%%%%%%
\begin{thm}
Let $u$ be the function obtained by Theorem \ref{conv-1}. Then it holds that 
\begin{align*}
u_n \to u \quad \text{weakly$^\ast$ in} \quad L^2(0,T;W^{2, \infty}(\Omega)) \quad \text{as} \quad n \to \infty. 
\end{align*} 
Moreover, if $N=1$, 
\begin{align*}
u_n \to u \quad \text{in} \quad C^{0,\vb}([\, 0,T \,];C^{1, \va}(\Omega)) \quad \text{as} \quad n \to \infty
\end{align*}
for every $\va \in (\, 0,1/2 \,)$ and $\vb \in (\, 0, (1-2\va)/8 \,)$, and if $N=2$, $3$, 
\begin{align*}
u_n \to u \quad \text{in} \quad C^{0,\vb}([\, 0,T \,];C^{0, \va}(\Omega)) \quad \text{as} \quad n \to \infty
\end{align*}
for every 
\begin{align*}
0 < \va <  2 - \dfrac{N}{2},  \qquad 0 < \vb <  \left( \dfrac{1}{2} - \dfrac{N}{8} \right) \left( 1 - \dfrac{\va}{2 - N/2} \right). 
\end{align*}
\end{thm}
%%%%%%%%%%%%%%%%%%%%%%%%%%%%%%%%

In order to complete the proof of Theorem \ref{main-thm}, we make use of the convergence result on the piecewise constant 
interpolation of $\{ u_{i,n}\}$. 
%%%%%%%%%%%%%%%%%%%%%%%%%%%%%%%%
\begin{lem}{\rm (\cite{NO})} \label{conv-piece-const}
Let $\tilde{u}_n$ be the piecewise constant interpolation of $\{ u_{i,n} \}$. If $N=1$, 
\begin{align*}
\tilde{u}_n \to u \quad \text{in} \quad L^\infty(0,T;C^{1,\va}(\Omega)) \quad \text{as} \quad n \to \infty 
\end{align*}
for every $\va \in (\, 0, 1/2 \,)$, where $u$ is the function obtained by Theorem \ref{conv-1}. 
If $N=2$, $3$, 
\begin{align*}
\tilde{u}_n \to u \quad \text{in} \quad L^\infty(0,T;C^{0,\va}(\Omega)) \quad \text{as} \quad n \to \infty 
\end{align*}
for every $\va \in (\, 0, 2-N/2 \,)$. Moreover, for any $N \ge 1$, it holds that 
\begin{align*}
\Delta \tilde{u}_n \rightharpoonup \Delta u \quad \text{in} \quad L^2(0,T;L^2(\Omega)) \quad \text{as} \quad n \to \infty.  
\end{align*} 
\end{lem}
%%%%%%%%%%%%%%%%%%%%%%%%%%%%%%%%

We are in a position to complete the proof of Theorem \ref{main-thm}. Let us define  
\begin{align}
\mu_n(t) = \mu_{i,n} \quad \text{if} \quad t \in (\, (i-1) \tau_n, i \tau_n \,], 
\end{align}
and set 
\begin{align}
\mC_f &= \{\, (x,t) \in (\Omega \setminus \Omega_0) \times \R_+ \mid u(x,t) = f(x) \,\}, \\
\mC_g &= \{\, (x,t) \in (\Omega \setminus \Omega_0) \times \R_+ \mid u(x,t) = g(x) \,\}. 
\end{align}

\smallskip

\noindent
{\it Proof of Theorem \ref{main-thm}}\,\,\,
Let $u$ be the function obtained by Theorem \ref{conv-1}. 
To begin with, along the same lines as in \cite{NO}, we see that 
\begin{align}
\int^T_0 \!\!\! \int_\Omega \left[ \pd_t u (w-u) + \Delta u \Delta (w-u) \right] \, dx dt \ge 0 \quad \text{for any} \quad w \in \mK,  
\end{align}
i.e., $u$ is a weak solution of \eqref{P}. Moreover the uniqueness follows from the results in \cite{B}. 

By virtue of Theorem \ref{Radon-2}, we deduce that 
\begin{align*}
\int^T_0 \mu_n(\Omega)^2 \, dt = \sum^{n}_{i=1} \int^{i \tau_n}_{(i-1) \tau_n} \mu_{i,n}(\Omega)^2 \, dt  
= \tau_n \sum^{n}_{i=1} \mu_{i,n}(\Omega)^2 
< C E(u_0) + T \| \Delta^2 f \|_{L^{\infty}(\Omega_0)}^2. 
\end{align*}
Thus, as $n \to \infty$, 
\begin{align*}
\mu_n \rightharpoonup \bar{\mu} \quad \text{weakly in} \quad L^2(0,T; \mM(\Omega)), 
\end{align*}
i.e., 
\begin{align*}
\int^T_0 \!\!\! \int_\Omega \vp\, d \mu_n dt \to \int^T_0 \!\!\! \int_\Omega \vp\, d \bar{\mu} dt 
\quad \text{for any} \quad \vp \in L^2(0,T; C^\infty_c(\Omega)) \quad \text{as} \quad n \to \infty. 
\end{align*}

Since $\mu_{i,n} \lfloor_{\Omega_0} = \Delta^2 f$, we observe from the definition of $\mu_n$ that 
\begin{align*}
\mu_n(t) \lfloor_{\Omega_0} = \Delta^2 f \quad \text{in} \quad [\, 0,T \,) \quad \text{for any} \quad n \in \N. 
\end{align*}
From now on, we set $\mu_n \lfloor_{\Omega \setminus \Omega_0} = \nu^{+}_{n} - \nu^{-}_{n}$ with 
\begin{align*}
\nu^{\pm}_{n}(t) = \mu^{\pm}_{i,n} \quad \text{if} \quad t \in (\, (i-1) \tau_n, i \tau_n \,],  
\end{align*}
where $\mu^{+}_{i,n}$ and $\mu^{-}_{i,n}$ denote respectively the positive part and the negative part of 
$\mu_{i,n} \lfloor_{\Omega \setminus \Omega_0}$. 
Since Theorem \ref{Radon-2} deduces that 
\begin{align*}
\int^T_0 \nu^{\pm}_{n}(\Omega)^2 \, dt < C, 
\end{align*}
there exist measures $\bar{\mu}_{\pm}$ such that 
\begin{align*}
\nu^{\pm}_{n} \rightharpoonup \bar{\mu}_{\pm} \quad \text{weakly in} \quad L^2(0,T;\mM(\Omega)) \quad 
\quad \text{as} \quad n \to \infty, 
\end{align*}
i.e., for any $\vp \in L^2(0,T; C^{\infty}_c(\Omega \setminus \Omega_0))$, 
\begin{align*}
\int^T_0 \!\!\! \int_\Omega \vp \, d \nu^{\pm}_{n} dt \to \int^T_0 \!\!\! \int_\Omega \vp \, d \bar{\mu}_{\pm} dt 
\quad \text{as} \quad n \to \infty. 
\end{align*}
On the other hand, it holds that  
\begin{align*}
\int^T_0 \!\!\! \int_\Omega \vp \, d \nu^{\pm}_n dt 
& = \pm \int^T_0 \!\!\! \int_\Omega \left[ \Delta \tilde{u}_n \Delta \vp + V_n \vp \right] \, dx dt  \\
& \to \pm  \int^T_0 \!\!\! \int_\Omega \left[ \Delta u \Delta \vp + \pd_t u \vp \right] \, dx dt 
\quad \text{as} \quad n \to \infty. 
\end{align*}
Thus we infer that $\bar{\mu}_{\pm} = \pm(\Delta^2 u + \pd_t u)$. 
We claim that 
\begin{align} \label{mu-supp-1}
{\rm supp}\, \bar{\mu}_{+} \subset \mC_f, \quad {\rm supp}\, \bar{\mu}_{-} \subset \mC_g. 
\end{align}
We shall prove the former relation. Let $x_0 \in \Omega \setminus (\mC_f \cup \Omega_0)$. 
Since $u$ is continuous in $\Omega \times \R_+$, there exist an open set $W \subset \Omega \setminus \Omega_0$, 
$0< t_1 < t_2 < T$, and $\vd>0$ such that  
\begin{align*}
u(x,t) - f(x) > \vd \quad \text{in} \quad W \times (\, t_1, t_2 \,). 
\end{align*}
It follows from Lemma \ref{conv-piece-const} that there exists a constant $N>0$ such that 
\begin{align*}
\tilde{u}_n(x,t) - u(x,t) > - \dfrac{\vd}{2} \quad \text{in} \quad W \times (\, t_1, t_2 \,) \quad \text{for any} \quad n \ge N, 
\end{align*}
so that 
\begin{align*}
\tilde{u}_n(x,t) - f(x) > \dfrac{\vd}{2} \quad \text{in} \quad W \times (\, t_1, t_2 \,) \quad \text{for any} \quad n \ge N. 
\end{align*}
This means that, for any $n \ge N$,  
\begin{align} \label{supp-rela-1}
W \times (\, t_1, t_2 \,) \subset \Omega \setminus (\mC^{+}_{i,n} \cup \Omega_0) 
\quad \text{for each} \quad \left[ \frac{t_1}{\tau_n}\right] \le i \le  \left[ \frac{t_2}{\tau_n} \right]. 
\end{align}
Thus we deduce that for any $\vp \in C_c((\, t_1,t_2 \,);C^{\infty}_{c}(W))$ 
\begin{align*}
\int^T_0 \!\!\! \int_\Omega \vp \, d \bar{\mu}_+ dt 
 &= \lim_{n \to \infty} \int^T_0 \!\!\! \int_\Omega \vp \, d \nu^{+}_{n} dt 
  = \lim_{n \to \infty} \int^T_0 \!\!\! \int_\Omega \left[ \Delta \tilde{u}_n \Delta \vp + V_n \vp \right] \, dx dt \\
 &= \lim_{n \to \infty} \sum^{n}_{i=1} 
        \int^{i \tau_n}_{(i-1) \tau_n}  \! \int_\Omega \left[ \Delta u_{i,n} \Delta \vp + V_{i,n} \vp \right] \, dx dt =0. 
\end{align*}
The last equality follows from \eqref{supp-rela-1}. This implies the relation \eqref{mu-supp-1}. 
\qed

%%%%%%%%%%%%%%%%%%%%%%%%%%%%%%%%%%%%%%%%%%%%%%%%%%%%%%
%%%%%%%%%%%%%%%%%%%%%%%%%%%%%%%%%%%%%%%%%%%%%%%%%%%%%%
%%%%%%%%%%%%%%%%%%%%%%%%%%%%%%%%%%%%%%%%%%%%%%%%%%%%%%

%%%%%%%%%%%%%%%%%%%%%%%%%%%%%%%%%%%%%%%%%%%%%%%%%%%%%%
%%%%%%%%%%%%%%%%%%%%%%%%%%%%%%%%%%%%%%%%%%%%%%%%%%%%%%
%%%%%%%%%%%%%%%%%%%%%%%%%%%%%%%%%%%%%%%%%%%%%%%%%%%%%%

\end{document}